\documentclass[11pt, twoside, fullpage]{article}
\usepackage{amssymb, amsmath, amsthm, verbatim}
\usepackage[top=30mm, left=23mm, right=23mm, bottom=30mm]{geometry}
\usepackage{inputenc}
\usepackage{titling}
\usepackage{fancyhdr}
\usepackage{placeins}
\usepackage{tikz}
\usepackage{verbatim}
\usetikzlibrary{positioning}
\usepackage{titling}
\usepackage{subcaption}
\usepackage{caption}
\usepackage[hyphens]{url}
\usepackage{hyperref}
\usepackage[hyphenbreaks]{breakurl}
\usepackage[T1]{fontenc}
\usepackage{currvita}
\usepackage{bbm}
\usepackage{enumitem}
\predate{}
\postdate{}
\usepackage{lipsum}
\newtheorem{theorem}{Theorem}
\newtheorem{corollary}[theorem]{Corollary}
\newtheorem{lemma}[theorem]{Lemma}
\newtheorem{proposition}[theorem]{Proposition}

\theoremstyle{definition}
\newtheorem{definition}{Definition}
\theoremstyle{remark}

\newcommand{\cA}{\mathcal{A}}

\newcommand{\cQ}{\mathcal{Q}}
\newcommand{\cP}{\mathcal{P}}

\newcommand{\cF}{\mathcal{F}}

\newcommand{\cV}{\mathcal{V}}

\begin{document}

\newcommand{\Addresses}{
\bigskip
\footnotesize

\noindent James~Brownlie, \textsc{Trinity College, University of Cambridge, CB2 1TQ, United Kingdom.}\par\noindent\nopagebreak\textit{Email address: }\texttt{jb2432@cantab.ac.uk}

\medskip

\noindent Sean~Jaffe, \textsc{Trinity College, University of Cambridge, CB2 1TQ, United Kingdom.}\par\noindent\nopagebreak\textit{Email address: }\texttt{scj47@cam.ac.uk}}

\pagestyle{fancy}
\fancyhf{}
\fancyhead [LE, RO] {\thepage}
\fancyhead [CE] {JAMES BROWNLIE AND SEAN JAFFE}
\fancyhead [CO] {THE INDUCED SATURATION NUMBER FOR $\mathcal{V}_3$ IS LINEAR}
\renewcommand{\headrulewidth}{0pt}
\renewcommand{\l}{\rule{6em}{1pt}\ }
\title{\Large{\textbf{THE INDUCED SATURATION NUMBER FOR \(\mathcal{V}_3\) IS LINEAR}}}
\author{JAMES BROWNLIE AND SEAN JAFFE}
\date{ }
\maketitle
\begin{abstract}

Given a poset $\mathcal{P}$, a family $\mathcal{F}$ of elements in the Boolean lattice is said to be $\mathcal{P}$-saturated if $\mathcal{F}$ does not contain an induced copy $\mathcal P$, but every proper superset of $\mathcal{F}$ contains one. The minimum size of a $\mathcal P$-saturated family in the $n$-dimensional Boolean lattice is denoted by $sat^*(n,\mathcal{P})$.\par In this paper, we consider the poset $\mathcal V_3$ (the four element poset with one minimal element and three incomparable maximal elements) and show that $sat^*(n,\mathcal{V}_3)\geq \frac{n}{2}$. This represents the first linear lower bound for $sat^*(n,\mathcal{V}_3)$, improving upon the previously best-known bound of $2\sqrt{n}$. Our result establishes that $sat^*(n,\mathcal{V}_3) = \Theta(n)$.

\end{abstract}

\section{Introduction}
A partially ordered set, or poset for short, is a set $\mathcal P$ equipped with a binary relation $\leq$ that is transitive, antisymmetric and reflexive. We say that a poset \((\mathcal{Q}, \leq')\) contains an induced copy of a poset \((\mathcal{P}, \leq)\) if there exists an injective function \(f: \mathcal{P} \to \mathcal{Q}\) such that \((\mathcal{P}, \leq)\) and \((f(\mathcal{P}), \leq')\) are isomorphic as posets. That is, for $u,v\in \cP$, we have $u\leq v $ if and only if $f(u)\leq' f(v)$. If a poset $\cQ$ does not contain an induced copy of a poset $\cP$, then we may say that $\cQ$ is $\cP$-free. Additionally, we say that a family \(\mathcal{F}\) of subsets of \([n]\) is \(\mathcal{P}\)-saturated if \(\mathcal{F}\) is \(\mathcal{P}\)-free, but \(\mathcal{F}\cup \{S\}\) is not for any $S\in \cP([n])\backslash \cF$. We write \(sat^*(n, \mathcal{P})\) for the minimum size of a \(\mathcal{P}\)-saturated family in \(\cP([n])\), and call this quantity the induced saturation number of \(\mathcal{P}\).

The study of saturation has a rich history. It was first considered in the context of graphs by by Erd\H{o}s, Hajnal and Moon~\cite{erdos1964problem} in 1964. This later inspired Gerbner, Keszegh, Lemons, Palmer, P{\'a}lv{\"o}lgyi and Patk{\'o}s \cite{gerbner2013saturating} to investigate saturation for posets. Later  Ferrara, Kay, Krammer, Martin, Reiniger, Smith and Sullivan~\cite{ferrara2017saturation} considered induced saturation for posets in 2017. The definitions in poset saturation are closely analogous to those used in the study of graph saturation. We note that the analogue to induced poset saturation in graphs is much less natural to consider than induced saturation for posets as it requires the introduction of tri-graphs \cite{chudnovsky2006berge}.


Arguably the most famous and important conjecture pertaining to induced poset saturation is that for any poset $\cP$, either ${sat}^*(n,\cP)$ is bounded (as $n \to \infty$), or $sat^*(n,\cP) = \Theta(n)$. More precisely, it is believed that if $sat^*(n,\cP)$ is unbounded then $sat^*(n,\cP)\geq n+1$. Partial progress towards this conjecture was first made by Keszegh, Lemons, Martin, Pálvölgyi, and Patkós \cite{keszegh2021induced}, who proved that the saturation number of any poset is either bounded or at least $\log_2(n)$. This was later improved by Freschi, Piga, Sharifzadeh, and Trelown \cite{freschi2023induced}, who showed that $sat^*(n,\cP)$ is either bounded or at least $2\sqrt{n}$. Progress towards a general upper bound was made by Bastide, Groenland, Ivan, and Johnston \cite{polynomial} showed that the saturation number of any poset can be bounded above by a polynomial.

Beyond these general bounds, there are very few posets for which the saturation number has been shown to grow linearly.  Liu \cite{liu2025induced} introduced arguably the widest class of posets $\cP$ that satisfy $sat^*(n,\cP)= \Omega(n)$. Given two finite posets $\mathcal P$ and $\mathcal Q$, we form the poset $\mathcal P*\mathcal Q$ by putting a copy of $\mathcal P$ entirely on top of a copy of $\mathcal Q$. This is also know as the linear sum of $\mathcal P$ and $\mathcal Q$. Liu proved that any poset of the form $\cP*\cA_2$ satisfies $sat^*(n,\cP*\cA_2)\geq n+1$, where $\cA_2$ represents the antichain of size two. By symmetry, we also have $sat^*(n,\cA_2*\cP)\geq n+1$. A similar result was proven by Ivan and Jaffe \cite{ivan2025saturationnumberdiamondlinear}, who showed that if $\cP_1$ and $\cP_2$ are non-empty posets such that $\cP_1 $ doesn't have a unique minimal element and $\cP_2$ doesn't have a unique maximal element, then $sat^*(n,\cP_1*\cA_2*\cP_2) = \Omega(n)$. 

Excluding these two relatively narrow results, only a very small number of posets $\cP$ are known to satisfy $sat^*(n,\cP)= \Omega(n)$. Notable examples include the antichain \cite{bastide2024exact} and $2C_2$ (the incomparable union of two chains of length two) \cite{martin2025induced}.

In this paper, we study the four element poset with one minimal element and three incomparable maximal elements. This poset, which we denote $\cV_3$, is displayed in Figure 1 below. By symmetry, $sat^*(n,\cV_3) = sat^*(n,\Lambda_3)$ so our analysis of $\cV_3$ also applies to $\Lambda_3$ (the poset displayed in Figure 2).

\begin{figure}[hbt!]
    \centering
    \begin{subfigure}[b]{0.4\textwidth}
        \centering
         \begin{tikzpicture}
            \node (top) at (2,0) {$\bullet$}; 
            \node (left) at (0,0) {$\bullet$};
            \node (right) at (4,0) {$\bullet$};
            \node (bottom) at (2,-2) {$\bullet$};
            \draw  (left) -- (bottom) -- (right);
            \draw (bottom)-- (top);
        \end{tikzpicture}
        \captionsetup{labelformat=empty}
        \caption{\textbf{Figure 1:} \(\cV_3\)}
    \end{subfigure}
    \hfill
    \begin{subfigure}[b]{0.4\textwidth}
        \centering
        \begin{tikzpicture}
            \node (top) at (2,2) {$\bullet$}; 
            \node (left) at (0,0) {$\bullet$};
            \node (right) at (4,0) {$\bullet$};
            \node (bottom) at (2,0) {$\bullet$};
            \draw (top) -- (left);
            \draw (right) -- (top) -- (bottom);
        \end{tikzpicture}
        \captionsetup{labelformat=empty}
        \caption{\textbf{Figure 2:} \(\Lambda_3\)}
    \end{subfigure}
\end{figure}
\FloatBarrier
Although studied as early as 2017 \cite{ferrara2017saturation}, even the order of magnitude of the growth of their saturation numbers was not known until now. In fact, $\cV_3$ and $\Lambda_3$ arguably represent the two simplest posets for which the order of growth of their induced saturation numbers was previously unknown.

It is easy to see that ${sat}^*(n,\mathcal{V}_3) \leq 2n$, since the union of two full chains in $\cP([n])$ that are disjoint excluding at $\emptyset$ and $[n]$ forms a $\mathcal{V}_3$-saturated subset of the Boolean lattice. As for the lower bound, the first non-trivial result was established by Ferrara, Kay, Kramer, Martin, Reiniger, and Smith \cite{ferrara2017saturation}, who showed that ${sat}^*(n,\mathcal{V}_3) \geq \log_2(n)$. The only subsequent improvement came from Freschi, Piga, Sharifzadeh, and Trelown \cite{freschi2023induced}, who proved that ${sat}^*(n,\mathcal{V}_3) \geq 2\sqrt{n}$. We note that these bounds are not specific to $\mathcal{V}_3$, but follow from the fact $sat^*(n,\cV_3)$ is unbounded. 

In this paper, we prove that \({sat}^*(n, \mathcal{V}_3) \geq \frac{n}{2}\), which establishes that both \(sat^*(n, \mathcal{V}_3)\) and \(sat^*(n, \Lambda_3)\) are \(\Theta(n)\). This allows us to lower bound the saturation number for a wide range of other posets. 
\section{Main results}
As mentioned previously, Theorem \ref{maintheorem} is the main result of the paper.
\begin{theorem}\label{maintheorem}
    \(sat^*(n,\cV_3)\geq \frac{n}{2}\).
\end{theorem}
By symmetry, we also have Corollary \ref{corollarytwo}:
\begin{corollary}\label{corollarytwo}
    \(sat^*(n,\Lambda_3) \geq \frac{n}{2}\).
\end{corollary}

Furthermore the Theorem \ref{maintheorem} and Corollary \ref{corollarytwo}, combined with Proposition 5 in \cite{gluing}, can be used to show that a large class of posets have saturation number that is at least linear.
\begin{proposition}[Proposition 5 in \cite{gluing}]
Let $\mathcal P_1$ and $\mathcal P_2$ be any non-empty posets such that $\mathcal P_1$ does not have a unique maximal element and $\mathcal P_2$ does not have a unique minimal element. Then $sat^*(n,\mathcal P_2*\mathcal P_1)\geq\max\{sat^*(n,\mathcal P_2*\bullet), sat^*(n,\bullet*\mathcal P_1)\}$, where $\bullet$ represents the poset with one element.
\end{proposition}
By a straightforward application of the above proposition, we find the following.
\begin{corollary}\label{pepepopocorolary}
Let $\mathcal P$ be any non-empty poset that does not have a unique maximal element. Then $sat^*(n,\mathcal A_3*\mathcal P) \geq sat^*(n,\mathcal{V}_3)\geq\frac{n}{2}$, where $\mathcal A_3$ is the antichain of size 3.

Likewise, if  $\mathcal P$ is a non-empty poset that does not have a unique minimal element then \break${sat^*(n,\mathcal P*\mathcal{A}_3) } \geq sat^*(n,\Lambda_3)\geq\frac{n}{2}$.
\end{corollary}

\section{Preliminary definitions}
In this section, we introduce some definitions that will be used in the proof of Theorem \ref{maintheorem}.
\begin{definition}
    Let \(\cF\subseteq \cP([n])\) and suppose $X,Y\in \cF$. We say that $X$ covers $Y$ in \(\cF\) if \(X\supsetneq Y\) and for any $Z\in \cF$, \(Y\subseteq Z\subseteq X\) implies that $Z = Y$ or $X$. When $\cF$ is implicit we may simply say that $X$ covers $Y$.
\end{definition}

\begin{definition} Let \(X,Y\in \cP([n])\) be such that $X\supsetneq Y$. Define \(\Theta_{X,Y}\) to be the set of triples $(A,B,C)$ of elements of $\cP([n])$ such that \(A\), \(B\), \(C\), \(X\) and \(Y\) induce the following Hasse diagram:
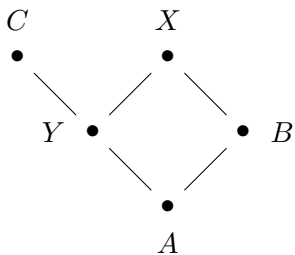
\begin{figure}[h]
\hspace{0.1cm}\\
\centering
\begin{tikzpicture}
\node[label=above:$X$] (X) at (6,1) {$\bullet$}; 
\node[label=right:$B$] (B) at (7,0) {$\bullet$};
\node[label=above:$C$] (C) at (4,1) {$\bullet$};
\node[label=below:$A$] (A) at (6,-1) {$\bullet$};
\node[label=left:$Y$] (Y) at (5,0) {$\bullet$};
\draw (X)  -- (Y);
\draw (B)-- (A);
\draw (C) -- (Y) -- (A);
\draw (B) -- (X);
\end{tikzpicture}
\end{figure}
\FloatBarrier
\end{definition}

\section{The proof of Theorem \ref{maintheorem}}

Let \(\cF\) be a \(\cV_3\)-saturated family. We will prove that $|\cF|\geq \frac{n}{2}$.

The first goal in this proof is to show that there exists an inclusion minimal element $S\in\cF$ such that $|S|\geq n - |\cF|$. We do this by constructing a chain \([n] = X_1\supsetneq X_2\supsetneq \dots\supsetneq X_k\) of elements of $\cF$ such that for each $i\in [n]\backslash X_k$, there exists a $Z_i\in\cF$ which is distinct because $Z_i\backslash X_\alpha = \{i\}$ for some $\alpha$. After this, we prove that every inclusion-minimal element of $\cF$ has size at most $|\cF|$. Combining these two facts gives the desired result.

\begin{lemma}\label{lemmatwo}

Suppose $X,Y\in \cF$ and $X$ covers $Y$. If \(i\in X\backslash Y\) then either there exists an element $Z_i\in \cF$ such that $Z_i\backslash Y = \{i\}$ or there are elements $A,B,C\in \cF$ such that \((A,B,C)\in \Theta_{X,Y}\).

\end{lemma}
\begin{proof}We may assume that $|X\backslash Y|\geq 2$ as if $|X\backslash Y| = 1$ and $i\in X\backslash Y$ then $X\backslash Y = \{i\}$. Setting $X = Z_i$ would complete the proof of the lemma in this case.

Let \(i\in X\backslash Y\). Since \(|X\backslash Y|\geq 2\), we have $Y\subsetneq Y\cup \{i\}\subsetneq X$. This means that $Y\cup \{i\}\not\in \cF$ as $X$ covers $Y$. Hence \(\cF\cup \{Y\cup \{i\}\}\) contains a copy of $\cV_3$ and furthermore $Y\cup \{i\}$ is an element of any such $\cV_3$. Let \(A,B,C\in \cF\) be elements such that $Y\cup \{i\},A,B,C$ forms a copy of $\cV_3$. As in the following diagram, we cannot have $Y\cup \{i\}$ being the minimal element of the $\cV_3$ as then $Y, A,B,C$ would form a copy of $\cV_3$ in $\cF$, a contradiction.
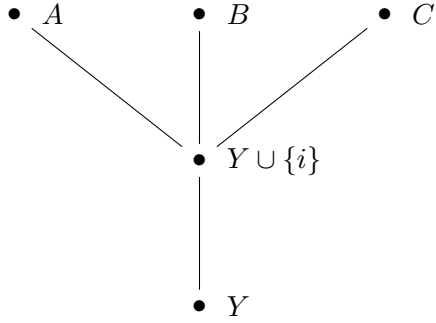
\begin{figure}[h]
\hspace{0.1cm}\\
\centering
\begin{tikzpicture}[
    node distance=1.5cm and 2cm,   
  ]

  \node[label=right:$A$]               (u1)    {$\bullet$};
  \node[right=of u1,label=right:$B$]  (u2)    {$\bullet$};
  \node[right=of u2,label=right:$C$] (ukm1) {\(\bullet\)};

  \node[below=of u2,label=right:$Y\cup\{i\}$]   (x)     {$\bullet$};

  \foreach \v in {u1,u2,ukm1}{
    \draw (x) -- (\v);
  }
  \node[below=of x,label=right:$Y$]   (y)     {$\bullet$};
\draw (x) -- (y);
\end{tikzpicture}
\end{figure}
\FloatBarrier
Hence \(Y\cup \{i\}\) is one of the elements in the top layer of the poset. Without loss of generality, let $A$ be the unique minimal element of the $\cV_3$. 

Since \(\cF\) is \(\cV_3\)-free, $X,A,B,C$ cannot be a copy of $\cV_3$. We know that $A\subsetneq B,C,X$ and $B\|C$, so either $X\not\|B$ or $X\not\|C$. Without loss of generality, suppose $X\not\|B$. If \(X\subseteq B\) then $Y\cup \{i\}\subseteq X\subseteq B$ which is a contradiction as $Y\cup \{i\}$ and $B$ are incomparable. This means that $B\subsetneq X$. Furthermore, if $B\supsetneq Y$ then we would have $X\supsetneq B\supsetneq Y$ which contradicts the fact that $X$ covers $Y$. Also, if $B\subseteq Y$ then $Y\cup\{i\}\supsetneq Y\supseteq B$ which contradicts $Y\cup\{i\}\| B$. Therefore we must have $B\| Y$. This is displayed in the following diagram, where the dashed lines represent relations that we don't currently know.
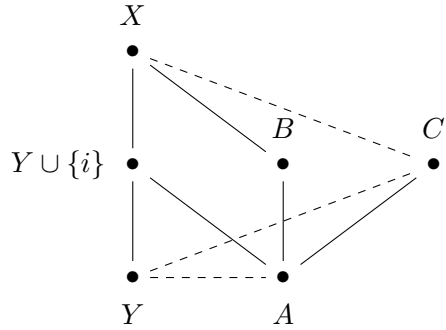
\begin{figure}[h]
\hspace{0.1cm}\\
\centering
\begin{tikzpicture}
\node[label=above:$X$] (top) at (6,1.5) {$\bullet$}; 
\node[label=left:$Y\cup \{i\}$] (left) at (6,0) {$\bullet$};
\node[label=above:$B$] (right) at (8,0) {$\bullet$};
\node[label=above:$C$] (right2) at (10,0) {$\bullet$};
\node[label=below:$A$] (right3) at (8,-1.5) {$\bullet$};
\node[label=below:$Y$] (bottom) at (6,-1.5) {$\bullet$};
\draw (top) -- (left) -- (bottom);
\draw (right)-- (right3) -- (right2);
\draw (left) -- (right3);
\draw (right) -- (top);
\draw[ dashed] (top) -- (right2);
\draw[dashed] (bottom) -- (right3);
\draw[ dashed] (bottom) -- (right2);
\end{tikzpicture}
\end{figure}
\FloatBarrier
Since $A\subseteq Y\cup \{i\}$, if \(A\) and $Y$ were incomparable we would have $i\in A$ and $A\backslash Y = \{i\}$. Hence in this case, we would be done by setting $Z_i = A$. As such, we may assume that either $A\subsetneq Y$ or $Y\subseteq A$. However, if $Y\subseteq A$ then we would have $Y\subsetneq B\subsetneq X$ so $X$ would not cover $Y$. This is a contradiction so we have $A\subsetneq Y$. 
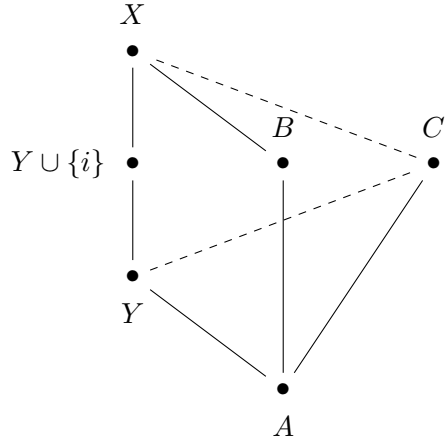
\begin{figure}[h]
\hspace{0.1cm}\\
\centering
\begin{tikzpicture}
\node[label=above:$X$] (X) at (6,1.5) {$\bullet$}; 
\node[label=left:$Y\cup \{i\}$] (left) at (6,0) {$\bullet$};
\node[label=above:$B$] (B) at (8,0) {$\bullet$};
\node[label=above:$C$] (C) at (10,0) {$\bullet$};
\node[label=below:$A$] (A) at (8,-3) {$\bullet$};
\node[label=below:$Y$] (Y) at (6,-1.5) {$\bullet$};
\draw (X) -- (left) -- (Y);
\draw (B)-- (A) -- (C);
\draw (Y) -- (A);
\draw (B) -- (X);
\draw[dashed] (Y)-- (C)-- (X);
\end{tikzpicture}
\end{figure}
\FloatBarrier
Finally, consider the relationship between $C$ and $Y$. Clearly, we cannot have $C\subseteq Y$ as $C$ and $Y\cup \{i\}$ are incomparable. Furthermore, if \(C\|Y\) then $Y,A,B,C$ would form a copy of $\cV_3$ in $\cF$, a contradiction. Therefore, we must have $Y\subsetneq C$. We cannot have $C\supseteq X$ as otherwise $C$ would be greater than $Y\cup \{i\}$. Furthermore, we do not have $C\subsetneq X$ as then $X$ would not cover $Y$. This means that $C\|X$, so we have the following Hasse diagram:

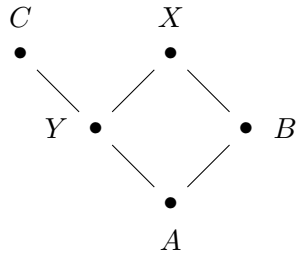
\begin{figure}[h]
\hspace{0.1cm}\\
\centering
\begin{tikzpicture}
\node[label=above:$X$] (X) at (6,1) {$\bullet$}; 
\node[label=right:$B$] (B) at (7,0) {$\bullet$};
\node[label=above:$C$] (C) at (4,1) {$\bullet$};
\node[label=below:$A$] (A) at (6,-1) {$\bullet$};
\node[label=left:$Y$] (Y) at (5,0) {$\bullet$};
\draw (X)  -- (Y);
\draw (B)-- (A);
\draw (C) -- (Y) -- (A);
\draw (B) -- (X);
\end{tikzpicture}
\end{figure}
\FloatBarrier
\end{proof}
\begin{proposition}\label{propositionthree}
    Suppose $X\neq Y\in \cF$ and $X$ covers $Y$. If there are elements $A,B,C\in \cF$ such that $(A,B,C)\in \Theta_{X,Y}$ then there exists an element $Y_2\in \cF$ such that $B\subseteq Y_2\subsetneq X$, $(A,Y_2,C)\in \Theta_{X,Y}$ and $X$ covers $Y_2$.
\end{proposition}
\begin{proof}
If \(X\) covers $B$, then we would be done by taking $B = Y_2$, so suppose that this is not the case. Let $Y_2\in \cF$ be an element such that $B\subsetneq Y_2\subsetneq X$ such that $X$ covers $Y_2$. We will prove that $(A,Y_2,C)\in \Theta_{X,Y}$. 

By construction, we have $Y_2\subsetneq X$. Since $A\subsetneq B$ and $B\subseteq Y_2$ we have $A\subsetneq Y_2$. We have $Y_2\not\subseteq C$ as otherwise we would have $B\subseteq Y_2\subseteq C$, violating the fact that $B\|C$. Likewise, if $Y_2\supseteq C$ then $X\supseteq Y_2\supseteq C$ which is a contradiction as $X\|C$. Therefore, we have $Y_2\|C$. $Y_2\not\subseteq Y$ because $Y_2\subseteq Y\Rightarrow B\subseteq Y$, which is a contradiction. Finally, if $Y_2\supsetneq Y$ then $X$ would not cover $Y$. This means that $Y\|Y_2$. Hence we have proven that \((A,Y_2,B)\in \Theta_{X,Y}\). 
\end{proof}
\begin{lemma}\label{bockandcalls}
    Let \(X\) be an element of $\cF$ that is not inclusion-minimal (i.e. $E\subsetneq X$ for some $E\in \cF$). There exists an element ${Y}\in \cF$ such that $X$ covers ${Y}$ and for every $i\in X\backslash {Y}$, there exists some $Z_i\in\cF$ such that $Z_i\backslash {Y} = \{i\}$.
\end{lemma}
\begin{proof}
Let \(X\) be an element of $\cF$ that is not inclusion-minimal. Suppose that for any $Y\in \cF$ which is covered by $X$, there exists a singleton $i\in X\backslash Y$ such that $Z\backslash Y \neq \{i\}$ for all $Z\in \cF$. 

By Lemma \ref{lemmatwo} and Proposition \ref{propositionthree}, if $Y\in \cF$ is covered by $X$, there exist elements $A,B,C\in\cF$ such that $(A,B,C)\in \Theta_{X,Y}$ and $X$ covers $B$. For every $Y\in \cF$ covered by $X$, let \((A_Y,B_Y,C_Y)\) be a triple of elements of $\cF$ such that   $(A_Y,B_Y,C_Y)\in \Theta_{X,Y}$ and $X$ covers $B_Y$ with $|C_Y|$ maximal. 

Pick $Y$ to be an element of $ \cF$ which is covered by $X$ with $|C_Y| $ maximal. Since $X$ covers \(B_Y\), we can define elements $A_{B_Y},B_{B_Y},C_{B_Y}\in \cF$ in the same way as before. 

   
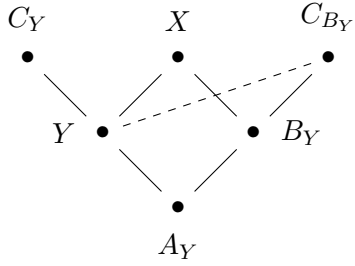
\begin{figure}[h]
\hspace{0.1cm}\\
\centering
\begin{tikzpicture}
\node[label=above:$X$] (X) at (6,1) {$\bullet$}; 
\node[label=right:$B_Y$] (B) at (7,0) {$\bullet$};
\node[label=above:$C_Y$] (C) at (4,1) {$\bullet$};
\node[label=below:$A_Y$] (A) at (6,-1) {$\bullet$};
\node[label=left:$Y$] (Y) at (5,0) {$\bullet$};
\node[label=above:$C_{B_Y}$] (coc) at (8,1) {$\bullet$};
\draw (X)  -- (Y);
\draw (coc) -- (B);
\draw (B)-- (A);
\draw (C) -- (Y) -- (A);
\draw (B) -- (X);
\draw[dashed] (coc) -- (Y);
\end{tikzpicture}
\end{figure}
\FloatBarrier

By definition, since \((A_Y,B_Y,C_Y)\in \Theta_{X,Y}\) and \((A_{B_Y},B_{B_Y},C_{B_Y})\in \Theta_{X,B_Y}\), we have $X\|C_Y$ and $X\|C_{B_Y}$. Furthermore $C_{B_Y}\not\subseteq C_Y$ as if this were the case then we would have $B_Y\subseteq C_{B_Y}\subseteq C_Y$, a contradiction as $B_Y\|C_Y$. By maximality of $|C_Y|$, we cannot have $C_Y\subsetneq C_{B_Y}$. This means that $C_Y\|C_{B_Y}$. 

Since \((A_Y,B_Y,C_Y)\in\Theta_{X,Y}\) we have $A_Y\subseteq X,B_Y,C_Y$. Thus, since $C_{B_Y}\supset B_Y$, $C_{B_Y}\supset A_Y$. This means that $C_Y$, $X$, $C_{B_Y}$ and $A_Y$ form a copy of $\cV_3$ in \(\cF\), which is a contradiction. Therefore, there must exist an element ${Y}\in \cF$ such that $X$ covers ${Y}$ and for every $i\in X\backslash {Y}$, there exists an element $Z_i$ such that $Z_i\backslash {Y} = \{i\}$. 
\end{proof}

\begin{lemma}\label{ejfioewjife2222}
There is an inclusion-minimal element $S\in \cF$ such that $|S|\geq n - |\cF|$.
\end{lemma}
\begin{proof}
    Since \(\cV_3\) does not have a unique maximal element, no copy of $\cV_3$ in $\cP([n])$ can use $[n]$. This means that $[n] \in \cF$. We construct a sequence \(X_1, X_2, \dots\) of elements of \(\mathcal{F}\) inductively as follows. Set \(X_1 = [n]\). For \(\alpha \geq 2\), suppose \(X_{\alpha-1}\) has already been defined and is not inclusion-minimal in \(\mathcal{F}\). Then we choose \(X_\alpha \in \mathcal{F}\) such that \(X_\alpha\) is covered by \(X_{\alpha-1}\) in \(\mathcal{F}\) and for every \(i \in X_{\alpha-1} \setminus X_\alpha\), there exists \(Z_i \in \mathcal{F}\) satisfying \(Z_i \setminus X_\alpha = \{i\}\). Lemma~\ref{bockandcalls} guarantees that such an \(X_\alpha\) always exists. If at any stage \(X_{\alpha-1}\) is an inclusion-minimal element of \(\mathcal{F}\), we terminate the sequence at \(X_{\alpha-1}\).

    We obtain a sequence $[n] = X_1\supsetneq X_2\supsetneq \dots\supsetneq X_k$ of elements of $\cF$ for some $k$. If $i\in [n]\backslash X_k$ then $i\in X_{\alpha-1}\backslash X_\alpha$ for some $\alpha$. This means that there exists an element $Z_i\in \cF$ such that $Z_i\backslash X_\alpha = \{i\}$. We will prove that each such $Z_i$ is unique.

    Suppose $i\neq j\in [n]\backslash X_k$ are such that $Z_i = Z_j$. There are elements $\alpha,\beta \in [k]$ such that $Z_i\backslash X_\alpha = \{i\}$ and $Z_j\backslash X_\beta = \{j\}$. Since \(X_1,X_2,\dots,X_k\) forms a chain, we may assume without loss of generality that $X_\alpha \subseteq X_\beta$. We have that $Z_i\subseteq X_\alpha \cup \{i\}\subseteq X_\beta\cup\{i\}$. Hence \(Z_i\backslash X_\beta \subseteq \{i\}\), so $\{j\}\subseteq \{i\}$. This implies that $j = i$ which is a contradiction. Therefore $Z_i\neq Z_j$ whenever $i\neq j\in [n]\backslash X_k$. This gives $n - |X_k|$ distinct elements of $\cF$. Hence \(|X_k|\geq n - |\cF|\). Since $X_k$ is inclusion-minimal, this completes the proof.
\end{proof}

\begin{lemma}\label{phockandcallsd}
   Let \(S\in \cF\) be an inclusion-minimal element. Then $|S|\leq |\cF|$. 
\end{lemma}
\begin{proof}
  We may assume that $S\neq \emptyset$ as otherwise the statement would trivially be true. 

  For each $i\in S$, we will exhibit an element $A\in \cF$ such that $S\backslash A = \{i\}$. This gives $|S|$ distinct elements of $\cF$ and thereby will complete the proof of the lemma.

  Suppose $i\in S$. By inclusion-minimality of $S$, $S\backslash\{i\}$ cannot be an element of $\cF$, so $\cF\cup \{S\backslash\{i\}\}$ must contain a copy of $\cV_3$ that uses $S\backslash\{i\}$ as an element. Let \(A,B,C\in \cF\) be such that when combined with $S\backslash \{i\}$ they form a copy of $\cV_3$. It is clear by minimality of $S$ that $S\backslash \{i\}$ must be the unique minimal element of the $\cV_3$. 

  Suppose $A\neq S$, $B\neq S$ and $C\neq S$. If \(i\in A\cap B\cap C\) then $S\subseteq A,B,C$ meaning that $A$, $B$, $C$ and $S$ form a copy of $\cV_3$ in $\cF$, a contradiction. Therefore we can assume without loss of generality that $i\not \in A$. Since \(S\backslash \{i\}\subseteq A\), we have $S\backslash A = \{i\}$ 

  On the other hand, if $C = S$ then $A$ and $S$ are incomparable and $S\backslash \{i\}\subseteq A$. This again means that $S\backslash A = \{i\}$.
\end{proof}
We now can complete the proof of Theorem \ref{maintheorem} by applying Lemmas \ref{ejfioewjife2222} and \ref{phockandcallsd}. By Lemma \ref{ejfioewjife2222}, there is an inclusion-minimal element $S\in \cF$ such that $|S|\geq n - |\cF|$. This means that $|\cF|\geq n - |\cF|$ so $|\cF|\geq \frac{n}{2}$.

\bibliographystyle{amsplain}
\bibliography{references}
\Addresses

\end{document}